\documentclass[letterpaper,10pt,conference]{ieeeconf}  
\IEEEoverridecommandlockouts                              
\usepackage{mathptmx} 
\usepackage{times} 
\usepackage{amsmath} 
\usepackage{amssymb}  

\usepackage{mathtools}

\mathtoolsset{centercolon}

\usepackage[noadjust,sort]{cite}

\title{\LARGE \bf Optimal Design of Networks of Positive Linear Systems\\under
Stochastic Uncertainty}

\author{Masaki Ogura and Victor M.~Preciado
\thanks{The authors are with the Department of Electrical and Systems
Engineering, University of Pennsylvania, Philadelphia, PA 19014, USA.
Email:  {\tt\small \{ogura,preciado\}@seas.upenn.edu}}%
\thanks{This work was supported in part by the NSF under grants CNS-1302222 and IIS-1447470.}%
}

\newtheorem{definition}{Definition}[section]

\newtheorem{lemma}[definition]{Lemma}
\newtheorem{proposition}[definition]{Proposition}
\newtheorem{theorem}[definition]{Theorem}
\newtheorem{remark}[definition]{Remark}

\newtheorem{problem}[definition]{Problem}

\usepackage{accents}
\newcommand{\ubar}[1]{\underaccent{\bar}{#1}}

\DeclareMathOperator*{\col}{col}
\renewcommand{\Pr}{P}

\DeclareSymbolFont{bbold}{U}{bbold}{m}{n}
\DeclareSymbolFontAlphabet{\mathbbold}{bbold}

\usepackage{calrsfs}
\DeclareMathAlphabet{\pazocal}{OMS}{zplm}{m}{n}
\renewcommand{\mathcal}[1]{\pazocal{#1}}

\usepackage{color}

\DeclareMathOperator*{\minimize}{minimize}
\DeclareMathOperator*{\subjectto}{subject\ to}
\DeclareMathOperator{\diag}{diag}

\DeclarePairedDelimiter\norm{\lVert}{\rVert}
\newcommand{\Norm}[1]{\left\lVert #1 \right\rVert}

\DeclareMathOperator{\Var}{Var}
\DeclareMathOperator{\esssup}{ess\,sup}

\usepackage{subcaption}

\newenvironment{proofof}[1]{
\renewcommand\proof{\noindent\hspace{2em}{\itshape Proof of #1: }}%
\begin{proof}}{\end{proof}
}

\usepackage{mathtools}

\allowdisplaybreaks[4]

\usepackage{flushend}

\begin{document}

\maketitle
\thispagestyle{empty}
\pagestyle{empty}

\begin{abstract}
In this paper, we study networks of positive linear systems subject to time-invariant and random uncertainties. We present linear matrix inequalities for checking the stability of the whole network around the origin with prescribed probability and decay rate. Based on this condition, we then give an efficient method, based on geometric programming, to find the optimal parameters of the probability distribution describing the uncertainty. We illustrate our results by analyzing the stability of a viral spreading process in the presence of random uncertainties.
\end{abstract}

\section{Introduction}

Stability analysis of uncertain dynamical systems or, precisely speaking,
dynamical systems with tine-invariant and uncertain parameters, has attracted
extensive attention for a long time in systems and control theory. In
particular, for the stability analysis of uncertain linear systems, there have
been proposed various tools including generalized Kharitonov's
theorem~\cite{Barmish1989}, common quadratic Lyapunov
functions~\cite{Bernussou1989}, parameter-dependent quadratic lyapunov
functions~\cite{Gahinet1996,Ramos2002}, and parameter-dependent polynomial
Lyapunov functions~\cite{DeOliveira1999,Oliveira2006,Lavaei2008}. For positive
linear systems, which are the class of linear systems whose state variable is
nonnegative entrywise provided the initial state is, the authors in
\cite{Briat2012c} and \cite{Colombino2015} propose robust stability conditions
based on diagonal Lyapunov functions~\cite{Shorten2009} and linear storage
functions, respectively.

All the works mentioned above model the uncertainty of a linear system by
providing a set of all possible configurations of the system. This is done
typically via polytopes where coefficient matrices of the system can belong to.
We can find in the literature several contribution to determine whether the
uncertain system is stable for all possible configurations or not. An important
consequence is that, in these frameworks, all possible system configurations are
equally important. However, in some applications, we are able to give not only
the set of possible uncertain parameters but also a weight, typically a
probability distribution, that measures the relative importance of the elements
in the uncertainty set.

One of the earliest works along this line is presented in~\cite{Spencer1994},
where the authors analyze the stability of linear time-invariant systems whose
coefficient matrices are modeled via a random vector. Using first- and
second-order reliability methods~\cite{Zhao1999}, the authors analyze the
stability of specific mechanical systems, and also study the probability of the
system being stable with respect to physical parameters, such as stiffness and
damping ratio. They also compare various control strategies such as those based
on linear quadratic regulator and Kalman filters. However, their results are
limited to the analysis of low-order systems and cannot be easily generalized to
the case of possibly large networks composed by uncertain systems.

In this paper, we present a robust stability analysis for networks of positive
linear systems subject to time-invariant and random uncertainty. We first
present linear matrix inequalities for checking if a given network of uncertain
positive linear systems is stable with a given probability and a decay rate.
Based on this result, we then present a convex optimization problem, posed as a
geometric program~\cite{Boyd2007}, for optimally designing the parameters of the
probability distribution expressing the uncertainty. We illustrate the obtained
results using a networked susceptible-infected-susceptible epidemic
model~\cite{VanMieghem2009a}, which has found applications in, for example,
public health~\cite{Pastor-Satorras2015}, malware spreading~\cite{Garetto2003},
and information propagation over socio-technical networks~\cite{Lerman2010}. The
results in this paper are based on the probabilistic estimate~\cite{Chung2011}
for the maximum real eigenvalue of random and symmetric matrices.

This paper is organized as follows. After introducing necessary notations, in
Section~\ref{sec:problemFormulation} we introduce our model of the network of
uncertain positive linear systems and then state the problems to be studied. We
then propose a convex optimization framework for analyzing the stability of
random networks in Section~\ref{sec:analysis}. Based on this analysis, we study
the optimal design of the probability distributions describing uncertainties in
Section~\ref{sec:design}. Numerical examples are presented in
Section~\ref{sec:example}.

\subsection{Mathematical Preliminaries} 

We denote by $\mathbb{R}$ the set of real numbers. The set~$\{1,\ldots,N\}$ is denoted by $[N]$. For vectors $x, y\in \mathbb{R}^N$, we write $x\geq
y$ ($x>y$) if $x_i \geq y_i$ ($x_i > y_i$, respectively) for every $i\in [N]$.
We say that $x$ is positive if $x>0$.  We denote the identity matrix by $I$. A
square matrix is said to be Metzler if its off-diagonal entries are nonnegative.
The Kronecker product of the matrices $A$ and $B$ is denoted by~$A\otimes B$.
The direct sum of the matrices~$A_1$, $\dotsc$, $A_N$, denoted
by~$\bigoplus_{i=1}^N A_i$, is defined as the block-diagonal matrix containing
the matrices ~$A_1$, $\dotsc$, $A_N$ as its diagonal blocks. When a symmetric
matrix~$A$ is positive semi-definite, we write $A\succeq 0$. For another
symmetric matrix~$B$, we write $A\succeq B$ if $A-B \succeq 0$. The notations
$A\succ B$ and $A\preceq B$ are defined in the obvious way. If $A \succeq 0$,
then $\sqrt A$ denotes a (not necessarily unique) matrix such that $A =
{\sqrt{A}} (\sqrt A)^\top$ holds. For a random matrix~$M$, we denote its
expectation by $E[M]$. The variance of $M$ is given by $\Var(M) =
E[(M-E[M])^2]$. Also we define the positive semi-definite matrix
\begin{equation*}
W(M) = E[M^\top M] - E[M]^\top E[M]. 
\end{equation*}
The symbol $\star$ is used to denote the symmetric blocks of partitioned
symmetric matrices.

The design framework proposed in this paper depends on a class of optimization
problems called geometric programs~\cite{Boyd2007}. Let $x_1$, $\dotsc$, $x_n$
denote positive variables and define ${x} = (x_1, \dotsc, x_n)$. We say that a
real-valued function $g({x})$ is a {\it monomial function} if there exist $c
\geq 0$ and $a_1, \dotsc, a_n \in \mathbb{R}$ such that $g({x}) = c
x_{\mathstrut 1}^{a_{1}} \dotsm x_{\mathstrut n}^{a_n}$. Also we say that a
function~$f({x})$ is a {\it posynomial function} if it is a sum of monomial
functions of ${x}$. For information about the modeling power of posynomial functions, we
point the readers to~\cite{Boyd2007}. Given a collection of posynomial functions
$f_0{(x)}$, $\dotsc$, $f_p{( x)}$ and monomials $g_1{(x)}$, $\dotsc$,
$g_q{(x)}$, the optimization problem
\begin{equation*} 
\begin{aligned}
\minimize_{ x}\ 
&
f_0({ x})
\\
\subjectto\ 
&
f_i({ x})\leq 1,\quad i=1, \dotsc, p, 
\\
&
g_j({ x}) = 1,\quad j=1, \dotsc, q, 
\end{aligned}
\end{equation*}
is called a {\it geometric program}. Although geometric programs are not convex,
they can be efficiently converted into a convex optimization problem.

\section{Problem formulation}\label{sec:problemFormulation}

In this section, we introduce the model of the network of linear systems with
stochastic uncertainty. We then state the problems studied in this paper.
Consider linear time-invariant and random systems
\begin{equation}\label{eq:Sigmai}
\Sigma_i : 
\frac{dx_i}{dt}  =  A_{ii} x_i + \sum_{j\neq i} A_{ij}x_j,\quad i=1, \dotsc, N, 
\end{equation}
where $A_{ij}$ is an $\mathbb{R}^{n\times n}$-valued random matrix for all $i, j
\in [N]$. We emphasize that $\Sigma_i$ is a linear time-invariant system for
each realization of the random matrices $A_{ij}$. In other words, the
coefficient matrices of $\Sigma_i$ do not change over time once they are chosen
from the corresponding distributions.

If we introduce the notations
\begin{equation*}
\begin{gathered}
x = \begin{bmatrix}
x_1 \\ \vdots \\ x_N
\end{bmatrix},\ A = \begin{bmatrix}
A_{11} & \cdots & A_{1N}
\\
\vdots & \ddots & \vdots
\\
A_{N1} & \cdots & A_{NN}
\\
\end{bmatrix}, 
\end{gathered}
\end{equation*}
then, from \eqref{eq:Sigmai}, we obtain the random linear time-invariant system
\begin{equation*}
\Sigma : 
\frac{dx}{dt} = Ax. 
\end{equation*}
Following the deterministic case~\cite{Farina2000}, we say that $\Sigma$ is \emph{positive} if $x(0) = x_0 \geq 0$ implies $x(t)\geq 0$ for every $t\geq 0$ with probability one. 

The first problem we address in this paper is the following stability
analysis problem with a prescribed decay rate and an unreliability level:

\begin{problem}[Stability analysis]
Given a desired decay rate $\lambda > 0$ and an unreliability level~$\epsilon\in
(0, 1]$, determine if, with probability at least $1-\epsilon$, the system
$\Sigma$ is stable with decay rate $\lambda$.
\end{problem}

We also investigate design problems. We assume that, though we cannot tune the
values of the random variables~$A_{ij}$ directly, we can still design their
probability distributions. Specifically, we assume that the probability
distributions of the random matrices $A_{ij}$ are parametrized by scalar
parameters $r_1$, $\dotsc$, $r_m$ that we can design. Our design problem is
based on cost functions and constraints. We suppose that there exists a function
$R(r, \epsilon)$ that represents the cost for realizing the specific parameter
$r$ and from  allowing the unreliability level~$\epsilon$. Also, for functions
$\theta_1$, $\dotsc$, $\theta_p$, $\phi_1$, $\dotsc$, $\phi_q$ of $r$ and
$\epsilon$, we allow the constraints on $r$ and $\epsilon$ of the form $f_k(r,
\epsilon)\leq 1$ ($k=1, \dotsc, p$) and $g_\ell(r, \epsilon) = 1$ ($\ell = 1,
\dotsc, q$). Now we can formulate the design problems studied in this paper:

\begin{problem}[Optimal design]
\label{prb:design:stbl} Given a desired decay rate $\lambda > 0$, a cost bound
$\bar R > 0$, and an unreliability level $\epsilon \in (0, 1]$, find the
parameter~$r$ such that the following conditions hold:
\begin{itemize}
\item With probability at least $1-\epsilon$, the system $\Sigma$ is stable with
decay rate $\lambda$;

\item The constraints $R(r, \epsilon) \leq \bar R$, $f_k(r,
\epsilon) \leq 1$ ($k\in [p]$), and $g_\ell(r, \epsilon) = 1$ ($\ell \in [q]$)
are satisfied.
\end{itemize} 
\end{problem}

For solving the above stated problems, we place one of the following
assumptions reflecting the networked-structure of the system $\Sigma$:
\begin{enumerate}
\item[A1)] The random variables $\{A_{ij}\}_{i, j\in [N]}$ are independent;

\item[A2)] The systems $\{\Sigma_i\}_{i\in [N]}$, i.e., the sets of random
variables $\{A_{i1}, \dotsc, A_{iN}\}_{i\in [N]}$, are independent.
\end{enumerate}

Finally, to the networked system $\Sigma$, we associate a directed graph
$(\mathcal V, \mathcal E)$ with $\mathcal V = [N]$ as follows. An ordered
pair~$(i, j)$, called a directed edge, is in $\mathcal E$ if $A_{ij}$ is not the
zero random variable.  We define the closed neighborhoods of $i$ by ${\mathcal
N}^-[i] = \{j\in[N] : (i, j) \in \mathcal E\}$ and ${\mathcal N}^+[i] =
\{j\in[N] : (j, i) \in \mathcal E\}$, respectively.

\section{Stability Analysis}\label{sec:analysis}

In this section, we present the solutions of the stability analysis problem. We
first state the results in Subsection~\ref{subsec:conditions}. The proof of the results is then presented in Subsection~\ref{subsec:proof}. 

\subsection{Stability Conditions}\label{subsec:conditions}

Throughout the paper, for an unreliability level~$\epsilon \in (0, 1]$, we define
\begin{equation*}
\rho= \log(nN/\epsilon).
\end{equation*}
The next theorem gives a solution for the stability analysis problem under
condition A1):

\begin{theorem}\label{thm:stabilityAnalysis:iid}
Suppose that A1) holds true. Let $\lambda> 0$ and $\epsilon \in (0, 1]$. Assume
that there exist positive numbers $p_1$, $\dotsc$, $p_N$, $a$, $\Delta$, and
$\sigma$ satisfying the linear matrix inequalities:
\begin{subequations}
\label{eq:stbl:analysis:iid}
\begin{gather}
E[A]^\top P + P E[A] + aI + \lambda P  \preceq 0, 
\label{eq:lmi:main}
\\
\label{eq:LMI:aDeltarho}
\begin{bmatrix}
a-\frac{\rho}{3}\Delta & \sqrt{2\rho} \sigma & \frac{\rho}{3}\Delta
\\
\star  & a-\frac{\rho}{3}\Delta & 0
\\
\star & \star& a-\frac{\rho}{3}\Delta 
\end{bmatrix}
\succ 0,
\\
p_i \esssup \norm{A_{ij} - E[A_{ij}]} \leq \Delta,
\label{eq:lmi:M:iid}
\\
\begin{bmatrix}
\sigma I &  p_i Q_i & R_i
\\
\star & \sigma I & O
\\
\star& \star& \sigma I
\end{bmatrix} \succeq 0, 
\label{eq:lmi:v:iid}
\end{gather}
\end{subequations}
where $P = \bigoplus_{i=1}^N (p_i I_n)$, and $Q_i, R_i$ are given for each $i\in [N]$ by
\begin{equation}\label{eq:defn:QiRi}
Q_i = \sqrt{\mathstrut\smash{\sum_{j \in \mathcal N^+[i]}} W(A_{ij}^\top)},\quad 
R_i = \col\bigl(p_j
\sqrt{\mathstrut W(A_{\smash{ji}})}\bigr)_{j \in \mathcal N^-[i]},
\vspace{.1cm}
\end{equation}
with $\col(\cdot)$ denoting the column vector obtained by stacking its
arguments. Then, with probability at least $1-\epsilon$, the system $\Sigma$ is
stable with decay rate $\lambda$. Moreover, if linear matrix inequalities
\eqref{eq:stbl:analysis:iid} are solvable for an $\epsilon = \epsilon_1$, then
so are for every $\epsilon \geq \epsilon_1$.
\end{theorem}

Several remarks on Theorem~\ref{thm:stabilityAnalysis} are in order. Though
feasibility of linear matrix inequalities~\eqref{eq:stbl:analysis:iid} implies
the stability of the \emph{averaged} system~$dx/dt = E[A] x$ with exponential
convergence rate~$\lambda$ due to \eqref{eq:lmi:main}, the converse is not
necessarily true mainly by the additional term $aI$ in the left hand side of
\eqref{eq:lmi:main}. Also notice that the coefficient $a$ of the additional is
related through \eqref{eq:LMI:aDeltarho} to $\Delta$ and $\sigma$, which
quantify the variability of the random system $\Sigma$ as can be seen from the
proof of the theorem.  Finally, the last claim of the theorem implies that a
bisection search effectively finds the minimum unreliability level given by
\begin{equation}\label{eq:minimumUnreliability}
\epsilon^\star = \inf\{ \epsilon>0 :
\text{\eqref{eq:stbl:analysis:iid} is solvable} \}.
\end{equation}

\newcommand{\clinnhd}[1]{\mathcal N_{#1}^-}
\newcommand{\cloutnhd}[1]{\mathcal N_{#1}^+}

\newcommand{\inlinebm}[1]{\setlength\arraycolsep{2pt}\begin{bmatrix}#1\end{bmatrix}}

Then we consider the condition A2). For each $i\in [N]$, define the
$\mathbb{R}^{n\times (nN)}$-valued random matrix $A_i = \inlinebm{A_{i1}
&\cdots&A_{iN}}$. Then, define the $\mathbb{R}^{(nN)\times (nN)}$-valued random
matrix $S_i = e_i^\top \otimes A_i^\top +  e_i \otimes A_i$ for each $i\in [N]$,
where $e_i$ denotes the $i$th standard unit vector in $\mathbb{R}^N$. Then, the
next theorem gives a solution to the stability analysis problem under condition
A2):

\begin{theorem}\label{thm:stabilityAnalysis}
Suppose that A2) holds true. Let $\lambda > 0$ and $\epsilon \in (0, 1]$. Assume
that there exist positive numbers $p_1$, $\dotsc$, $p_N$, $a$, $\sigma$, and
$\Delta$ satisfying the linear matrix inequalities:
\begin{subequations}
\label{eq:stbl:analysis}
\begin{gather}
\text{\eqref{eq:lmi:main} and \eqref{eq:LMI:aDeltarho}},
\\
2 p_i \esssup \norm{A_i - E[A_i]} \leq \Delta,
\label{eq:lmi:M}
\\
\begin{bmatrix}
\sigma  I & p_1 \sqrt{\Var(S_1)} & \cdots & p_N \sqrt{\Var(S_N)}
\\
\star & \sigma I &  & 
\\
\vdots & & \ddots
\\
\star& & & \sigma I
\end{bmatrix} \succeq 0. 
\label{eq:lmi:v}
\end{gather}
\end{subequations}
Then, with probability at least $1-\epsilon$, the system $\Sigma$ is stable with
decay rate $\lambda$. Moreover, if linear matrix inequalities
\eqref{eq:stbl:analysis} are solvable for an $\epsilon = \epsilon_1$, then so
are for every $\epsilon \geq \epsilon_1$.
\end{theorem}

\begin{remark}
By the structure of matrix $S_i$ and the definition of the
graph~$\mathcal V$, the positive semi-definite matrix~$\Var(S_i)$ has rank at
most $nd^-[i]$. Therefore, we can take the square root $\sqrt{\Var(S_i)}$ having at most $n d^-[i]$ rows. 
\end{remark}

\subsection{Proof} \label{subsec:proof}

For the proof of Theorems~\ref{thm:stabilityAnalysis:iid}
and~\ref{thm:stabilityAnalysis}, we recall the following probabilistic estimate
on the maximum eigenvalue of the  sum of random and symmetric matrices:

\begin{proposition}[\cite{Chung2011}]\label{prop:Chung} 
For two positive constants $\Delta$ and $\sigma$, define the function
\begin{equation*}
\kappa_{\Delta, \sigma^2}(a) = n \exp\left(
- \frac{a^2}{2\sigma^2 + \frac{2\Delta a}{3}}
\right),\ a\geq 0.
\end{equation*}
Let $X_1$, $\dotsc$, $X_N$ be independent random $n\times n$ symmetric matrices.
Let $\Delta$ be a nonnegative constant such that $\norm{X_i-E[X_i]} \leq \Delta$
for every $i\in [N]$ with probability one. Also take an arbitrary $\sigma \geq
0$ satisfying $\norm{\sum_{i=1}^N \Var(X_i)}\leq \sigma^2$. Then, the sum $X =
\sum_{i = 1}^N X_i$ satisfies
\begin{equation*}
\Pr\left(\eta(X) \geq \eta(E[X]) + a\right) < \kappa_{\Delta, \sigma^2}(a)
\end{equation*}
for every $a > 0$. 
\end{proposition}

About the function $\kappa_{\Delta, \sigma^2}$ appearing in this proposition, we can prove the following straightforward but yet important lemma: 

\begin{lemma}\label{lem:key}
For all positive numbers $\Delta$, $\sigma$, $a$, and $\epsilon$, the following
statements are equivalent:
\begin{enumerate}
\item \label{item:kappa} $\kappa_{\Delta, \sigma^2}(a) < \epsilon$; 

\item \label{item:gp} $2\rho \Delta a^{-1} + 6\rho \sigma^2 a^{-2} < 3$;

\item \label{item:lmi} The matrix inequality \eqref{eq:LMI:aDeltarho} holds.
\end{enumerate}
Moreover, for fixed $\Delta$, $\sigma$, and $a$, if one of the above equivalent
statements is satisfied by an $\epsilon = \epsilon_1$, then the statements are
satisfied for every $\epsilon \geq \epsilon_1$.
\end{lemma}

\begin{proof}
Taking the logarithm in the both hand sides of the inequality $\kappa_{\Delta,
\sigma^2}(a) < \epsilon$ immediately gives the equivalence
[\ref{item:kappa})~$\Leftrightarrow$~\ref{item:gp})]. Also, the equivalence
[\ref{item:gp})~$\Leftrightarrow$~\ref{item:lmi})] readily follows from taking
the Schur complement of the matrix in the inequality~\eqref{eq:LMI:aDeltarho}
with respect to its $(1, 1)$-entry. Then, the latter claim about the
monotonicity follows from the fact that the left hand side of the inequality in
\ref{item:gp}) is increasing with respect to $\rho$ and therefore decreasing
with respect to $\epsilon$.
\end{proof}

Let us prove Theorem~\ref{thm:stabilityAnalysis:iid}.

\begin{proofof}{Theorem~\ref{thm:stabilityAnalysis:iid}}
Assume that positive numbers $p_1$, $\dotsc$, $p_N$, $a$, $\Delta$, and $\sigma$
solve linear matrix inequalities~\eqref{eq:stbl:analysis:iid}. Let $X_{ij} =
p_i(U_{ji}\otimes A_{ij}^\top + U_{ij}\otimes A_{ij})$, where $U_{ij} \in
\mathbb{R}^{N\times N}$ denotes the $\{0,1\}$-matrix elements are all zero
except its $(i, j)$-entry. Then, from the basic property of Kronecker products
of matrices~\cite{Brewer1978}, we obtain $A^\top P + PA + \lambda P= \lambda P +
\sum_{i=1}^N\sum_{j=1}^N X_{ij}$, to which we apply
Proposition~\ref{prop:Chung}. By \eqref{eq:lmi:M:iid}, we can derive the
estimate
\begin{equation*}
\begin{aligned}
\norm{X_{ij} - E[X_{ij}]}
&=
\begin{multlined}[t][.3\textwidth]
p_i\Bigl\lVert
U_{ji}\otimes (A_{ij} - E[A_{ij}])^\top  +  \\ U_{ji} \otimes (A_{ij} - E[A_{ij}])
\Bigr\rVert
\end{multlined}
\\
&=
p_i \norm{A_{ij}- E[A_{ij}]} 
\\
&\leq \Delta.
\end{aligned}
\end{equation*}
Also, since a straightforward computation shows
\begin{equation*}
\Var(X_{ij}) = p_i^2 (U_{jj}\otimes (A_{ij}^\top A_{ij}) + 
U_{ii} \otimes (A_{ij} A_{ij}^\top)), 
\end{equation*}
we have
\begin{equation*}
\sum_{i=1}^N \sum_{j=1}^N \Var(X_{ij})
=
\sum_{i=1}^NU_{ii}\otimes \left(
\sum_{j=1}^N (p_j^2 W(A_{ji}) + p_i^2 W(A_{ij}^\top)
)
\right).
\end{equation*}
Therefore, by \eqref{eq:lmi:v:iid} and the definition \eqref{eq:defn:QiRi} of
the matrices $Q_i$ and $R_i$,
\begin{equation*}
\begin{aligned}
\Norm{\sum_{i=1}^N\sum_{j=1}^N \Var(X_{ij})}
&=\max_{1\leq i\leq N} \Norm{
p_{i}^2 \sum_{j=1}^N W(A_{ij}^\top) + \sum_{j=1}^N p_j^2 W(A_{ji})
}
\\
&=
\max_{1\leq i\leq N} \Norm{
(p_iQ_i)(p_iQ_i)^\top + R_i R_i^\top
}
\\
&\leq \sigma^2.
\end{aligned}
\end{equation*}
Now, by Proposition~\ref{prop:Chung}, we have
\begin{equation*}
\begin{multlined}[.45\textwidth]
\Pr\Bigl(\lambda_{\max}(A^\top P + PA + \lambda P) \geq\\ \lambda_{\max}(E[A]^\top P+ PE[A] + \lambda P) + a  \Bigr)
<
\kappa_{\Delta, \sigma^2} (a).
\end{multlined}
\end{equation*}
By \eqref{eq:LMI:aDeltarho} and Proposition~\ref{lem:key}, we have
$\kappa_{\Delta, \sigma^2} (a) < \epsilon$. Therefore, the inequality
\eqref{eq:lmi:main} implies that $\Pr\bigl(\lambda_{\max}(A^\top P + PA +
\lambda P) \geq  0\bigr) < \epsilon$. This shows that, with probability at least
$1-\epsilon$, we have that $A^\top P + PA + \lambda P < 0$. This means that,
with probability at least $1-\epsilon$, the system~$\Sigma$ has the Lyapunov
function $V(x) = x^\top P x$ with decay rate $\lambda$. This completes the proof
of the first part of the theorem.

Let us then prove the second statement of the theorem. Let $\epsilon_1>0$ be
arbitrary and assume that linear matrix
inequalities~\eqref{eq:stbl:analysis:iid} are solvable when $\epsilon =
\epsilon_1$ with the parameters  $p_1$, $\dotsc$, $p_N$, $a$, $\lambda$,
$\Delta$, and $\sigma$. We show that, with the same parameters,
inequalities~\eqref{eq:stbl:analysis:iid} are solvable whenever $\epsilon\geq
\epsilon_1$. By the choice of the parameters, all the linear matrix inequalities
in \eqref{eq:stbl:analysis:iid} except \eqref{eq:LMI:aDeltarho} hold true. Also,
the feasibility of \eqref{eq:LMI:aDeltarho} follows from the latter claim in
Lemma~\ref{lem:key}. This completes the proof of the theorem. 
\end{proofof}

\begin{remark}
As can be observed from the above proof, in
Theorem~\ref{thm:stabilityAnalysis:iid}, we confine our attention to the
Lyapunov functions $x^\top P x$ with diagonal $P$. This choice is motivated by
the following fact~\cite{Shorten2009}: a positive linear system $dx/dt = Ax$ is
stable if and only if it admits a Lyapunov function $x^\top P x$ with diagonal
$P \succ 0$. Also notice that we use the diagonal matrix~$P \in
\mathbb{R}^{(nN)\times (nN)}$ determined by only $N$ parameters $p_1$, $\dotsc$,
$p_N$. Using the fully parametrized diagonal matrix~$P$ would yield a less
conservative result than the theorem. In this paper, we however choose not to
present the fully parametrized case to keep presentation simple.
\end{remark}

We then give the proof of Theorem~\ref{thm:stabilityAnalysis}.

\begin{proofof}{Theorem~\ref{thm:stabilityAnalysis}}
Assume that positive numbers $p_1$, $\dotsc$, $p_N$, $a$, $\lambda$, $\sigma$,
and $\Delta$ solve linear matrix inequalities~\eqref{eq:stbl:analysis}. Let $X_i
= p_i S_i$. Then we have $A^\top P + PA + \lambda P = \lambda P +\sum_{i=1}^N
X_i$, to which we again apply Proposition~\ref{prop:Chung} as in the proof of
Theorem~\ref{thm:stabilityAnalysis:iid}. By \eqref{eq:lmi:M}, we can show that
\begin{equation*}
\begin{aligned}
\norm{X_{i} - E[X_{i}]}
&=
p_i\Norm{
e_i^\top \otimes (A_i - E[A_i])^\top  +  e_i \otimes (A_i - E[A_i])
}
\\
&\leq
2 p_i \norm{e_i \otimes (A_i - E[A_i])}
\\
&=
2p_i \norm{A_{i} - E[A_{i}]} 
\\
&
\leq \Delta.
\end{aligned}
\end{equation*}
Also, the inequality~\eqref{eq:lmi:v} immediately shows $\norm{\sum_{i=1}^N
\Var(X_i)} = \norm{\sum_{i=1}^N p_i^2 \Var(S_i)}\leq \sigma^2$. The rest of the
proof is the same that of Theorem~\ref{thm:stabilityAnalysis:iid} and hence is
omitted.
\end{proofof}

\section{Optimal Design}\label{sec:design}

Based on Theorems~\ref{thm:stabilityAnalysis:iid} and
\ref{thm:stabilityAnalysis}, in this section we study network design problems.
Roughly speaking, designing the distributions of $A$ or, the parameters $r_1$,
$\dotsc$, $r_m$, corresponds to solving matrix
inequalities~\eqref{eq:stbl:analysis:iid} or \eqref{eq:stbl:analysis} with
$E[A]$ being a \emph{variable}. This in particular makes the inequalities not
linear with respect to decision variables. To avoid the difficulty, in this
paper we employ geometric programming~\cite{Boyd2007} instead of linear matrix
inequalities. For this purpose, we place the following assumptions on the random
coefficient matrices of $\Sigma$:
\begin{enumerate}
\item[B1)] There exist random variables $A_+$ and $A_-$ satisfying $A = A_+ - A_-$
such that $E[A_+]$ is a posynomial matrix and $E[A_-]$ is a diagonal monomial
matrix in variables $r_1$, $\dotsc$, $r_m$;

\item[B2)] The cost function $R$ and the constraint functions $f_1$, $\dotsc$,
$f_p$ are posynomial functions in $r_1$, $\dotsc$, $r_m$, and $\rho$.

\item[B3)] The constraint functions $g_1$, $\dotsc$, $g_q$ are monomials in $r_1$, $\dotsc$, $r_m$, and $\rho$.
\end{enumerate}

Under this assumption, the next theorem gives a solution to the stabilization
problem for the case A1) holds: 

\begin{theorem}\label{thm:stabilization:iid}
Suppose that $\Sigma$ is positive and A1) holds true. For all $i, j\in [N]$, let
$\eta_i$, $\Phi_{ij}$ $\Psi_{ij}$ be posynomial functions in $r_1$, $\dotsc$, $r_m$ such
that
\begin{equation}\label{eq:def:Phi_i:iid}
\begin{gathered}
\esssup \norm{A_{ij} - E[A_{ij}]} \leq \eta_{ij}, 
\\
W(A_{ij}) \preceq \Phi_{ij}
,\quad 
W(A_{ij}^\top) \preceq \Psi_{ij}.
\end{gathered}
\end{equation}
Assume that the following geometric
program is feasible:
\begin{subequations}
\label{eq:gp:stbl:iid}
\begin{align}
\hspace{-.65cm}\minimize_{\substack{
a,\,\Delta,\,\sigma,\,\rho,\,\lambda \in \mathbb{R},\\
p\in\mathbb{R}^{N},\,r \in \mathbb{R}^m,\\v\in\mathbb{R}^{nN},\,w_i\in\mathbb{R}^n}}\,&1/\lambda
\\
\hspace{-.65cm}\subjectto\ 
&
(E[A_+^\top] P \!+\! P E[A_+] \!+\! aI \!+\! \lambda P)v  \leq 2PE[A_-] v,\hspace{-.5cm}
\label{eq:gp:main:iid}
\\
&
2\rho \Delta a^{-1} + 6\rho \sigma^2 a^{-2} < 3,
\label{eq:gp:s:iid}
\\
&
p_i\eta_{ij}(r) \leq \Delta,
\label{eq:gp:M:iid}
\\
&
\left(p_{i}^2 \sum_{j=1}^N \Psi_{ij}(r) + \sum_{j=1}^N p_j^2 \Phi_{ji}(r)\right)w_i \leq \sigma^2 w_i, \hspace{-.4cm}
\label{eq:gp:v:iid}
\\
&R(r, \epsilon(\rho)) \leq \bar R,\label{eq:gp:cost}\\
&f_k(r, \epsilon(\rho)) \leq 1,\ g_\ell(r, \epsilon(\rho))=1.
\label{eq:gp:constraints}
\end{align}
\end{subequations}
Let $\rho^\star$, $r^\star$, and $\lambda^\star$ be the optimal solution of this
optimization problem. Define $\epsilon^\star = nN/e^{-\rho^\star}$. Then, with
probability at least $1-\epsilon^\star$, the system $\Sigma$ with parameters
$r^\star$ is stable with decay rate $\lambda^\star$.
\end{theorem}

\begin{proof}
Conditions B1)--B3) guarantee the optimization problem~\eqref{eq:gp:stbl:iid} to
be a geometric program. Assume that the optimization problem
\eqref{eq:gp:stbl:iid} is feasible. We show that the linear matrix
inequalities~\eqref{eq:stbl:analysis:iid} are feasible. The
inequality~\eqref{eq:gp:main:iid} implies that $(E[A]^\top P + P E[A] + aI +
\lambda P)v \leq 0$. Since the matrix $E[A]^\top P + P E[A] + aI + \lambda P$ is
Metzler by the positivity of $\Sigma$, the Perron-Frobenius theory and
\eqref{eq:gp:main:iid} show that the matrix is negative semi-definite, i.e., the
linear matrix inequality~\eqref{eq:lmi:main} holds. By
Proposition~\ref{lem:key}, the inequality~\eqref{eq:gp:s:iid} equivalently
implies \eqref{eq:LMI:aDeltarho}. Also, the inequality~\eqref{eq:gp:M:iid} and
the definition of $\eta_{ij}$ show \eqref{eq:lmi:M:iid}. Furthermore, the
inequality \eqref{eq:gp:v:iid} and \eqref{eq:def:Phi_i:iid} imply
\eqref{eq:lmi:v:iid}. Hence, by Theorem~\ref{thm:stabilityAnalysis}, we obtain
the conclusion.
\end{proof}

Then, the next theorem gives a solution for the stabilization problem under A2):

\begin{theorem}
Suppose that $\Sigma$ is positive and A2) holds true. For each $i\in [N]$, let
$\eta_i$ and $\Phi_i$ be posynomial functions in $r$ such that
\begin{equation*} 
\esssup \norm{A_{i} - E[A_{i}]} \leq \eta_{i}, \quad 
\Var(S_i) \preceq \Phi_i(r), 
\end{equation*}
for every feasible $r$. Assume that the following geometric
program is feasible:
\begin{equation*}
\begin{aligned}
\minimize_{\substack{
a,\,\Delta,\,\sigma,\,\rho,\,\lambda \in \mathbb{R},\\
p\in\mathbb{R}^{N},\,r \in \mathbb{R}^m,\,v, w_i\in\mathbb{R}^{nN}}}\ &1/\lambda
\\
\subjectto\ \ \ \ \ \ 
&
\text{\eqref{eq:gp:main:iid}, \eqref{eq:gp:s:iid}, \eqref{eq:gp:cost}, and \eqref{eq:gp:constraints}},
\\
&
2 p_i \eta_i(r) \leq \Delta,
\\
&
\left(\sum_{i=1}^N p_i^2 \Phi_i(r)\right) w_i \leq \sigma^2 w_i.
\end{aligned}
\end{equation*}
Let $\rho^\star$, $r^\star$, and $\lambda^\star$ be the solutions of this
optimization problem. Define $\epsilon^\star = nN/e^{-\rho^\star}$. Then, with
probability at least $1-\epsilon^\star$, the system $\Sigma$ with parameter
$r^\star$ is stable with decay rate $\lambda^\star$.
\end{theorem}

\begin{proof}
The proof is almost the same as the proof of Theorem~\ref{thm:stabilization:iid}
and hence is omitted.
\end{proof}

\section{Numerical Example}\label{sec:example}

In this section, we illustrate the obtained results using a famous
disease-spreading model in epidemiology called the heterogeneous networked
susceptible-infected-susceptible model~\cite{VanMieghem2009a}. In the model, the
evolution of the disease in a networked population whose graph has the adjacency
matrix $A_G = [a_{ij}]_{i, j} \in \{0, 1\}^{N\times N}$ is described as
\begin{equation}\label{eq:SIS}
\frac{dx_i}{dt} = -\delta_i x_i +  \sum_{j=1}^N \beta_{ij}a_{ij}x_j, \quad
i\in [N], 
\end{equation}
where $\beta_{ij}$ and $\delta_i$ are positive constants. The variable $x_i(t)$
represents the probability that node $i$ is infected at time~$t$. The constant
$\beta_{ij}$, called the transmission rate, indicates the rate at which the
infection is transmitted to node $i$ from its infected neighbor~$j$. The
constant $\delta_i$, called the recovery rate, indicates the rate at which the
infection is cured. Define $B = [\beta_{ij}a_{ij}]_{i, j} \in
\mathbb{R}^{N\times N}$ and $D = \diag(\delta_1, \dotsc, \delta_N)$. Then, the
dynamics in \eqref{eq:SIS} are written by the differential equation
\begin{equation}\label{eq:B-Dsys}
\frac{dx}{dt} = (B-D) x, 
\end{equation}
whose stability indicates that the infection will be eradicated asymptotically.

Assume that we have an access to preventative resource that can change the
natural transmission rate, denoted by~$\bar \beta_{ij}$ to another rate~$\ubar
\beta_{ij}$ smaller than $\bar{\beta}_{ij}$. We consider the situation that,
though the preventative resource is expected to be applied to all the possible
edges in the network, only a fraction of them in fact take it. Let us model this
uncertainty as
\begin{equation*}
\beta_{ij} = 
\begin{cases}
\bar \beta_{ij} & \text{with probability $r_{ij}$},
\\
\ubar \beta_{ij} & \text{with probability $1- r_{ij}$},
\end{cases}
\end{equation*}
where $r_{ij} \in [0, 1]$ is a constant. We assume that the events of resource
being applied to edge $(i, j)$ are independent for all the edges.

We remark that this problem setting is motivated by imperfect vaccine coverage
commonly observed in human networks~\cite{Metcalf2014}. This problem is studied
in Magpantay et al.~\cite{Magpantay2014} under the assumption that $A$ is the
complete graph, i.e., the graph in which every pair of distinct vertexes is
connected by a unique edge. On the other hand, in this paper we allow the
adjacency matrix $A_G$ of the network to be arbitrary. We also remark that a
similar problem is considered in \cite{Preciado2014}, where the authors directly
design the values of transmission rates over an interval. In this paper, we are
considering a more realistic scenario when we can only give or not give only
one type of preventative resource.

\subsection{Stability Analysis}

We first solve the stability analysis problem. For simplicity, we assume that
all the $r_{ij}$ share the same value $r$. Also, we assume that the share the
same natural transmission rate and the transmission rate after prevention as
$\bar{\beta}_{ij} = \bar{\beta}$ and $\ubar{\beta}_{ij} = \ubar{\beta}$ for
positive numbers $\ubar{\beta}$ and $\bar \beta$. Then, we can find $Q_i$ and
$R_i$ given in \eqref{eq:defn:QiRi} by using $W(A_{ii}) = 0$ and $W(A_{ij}) =
r(1-r)(\bar{\beta} - \ubar{\beta})^2$ if $i\neq j$.

We let $A_G$ be a realization of the directed Erd\H{o}s-R\'enyi graph with $N =
200$ nodes and diedge probability $p = 0.05$. We use the parameters $\delta=1$,
$\bar{\beta} = 1.1/\lambda_{\max}(A)$, and $\ubar{\beta} =
0.1/\lambda_{\max}(A)$ for all $i, j\in [N]$. Notice that, since the matrix
$\bar \beta A - D$ has a positive eigenvalue $\bar{\beta}\lambda_{\max}(A) -
\delta = 0.1$, the system~\eqref{eq:B-Dsys} is not stable if no preventative
resource is applied to any edge. For various values of $\lambda$ and $r$, we
calculate the minimum unreliability rate~$\epsilon^\star$ given
in~\eqref{eq:minimumUnreliability}. Fig.~\ref{fig:analysis} shows the obtained
values of $\epsilon^\star$. We can see that, the smaller the non-prevention
rate~$r$ is, with the larger probability we can guarantee the stability of
$\Sigma$ with the larger decay rate.

\begin{figure}
\centering
\includegraphics[width=.475\textwidth]{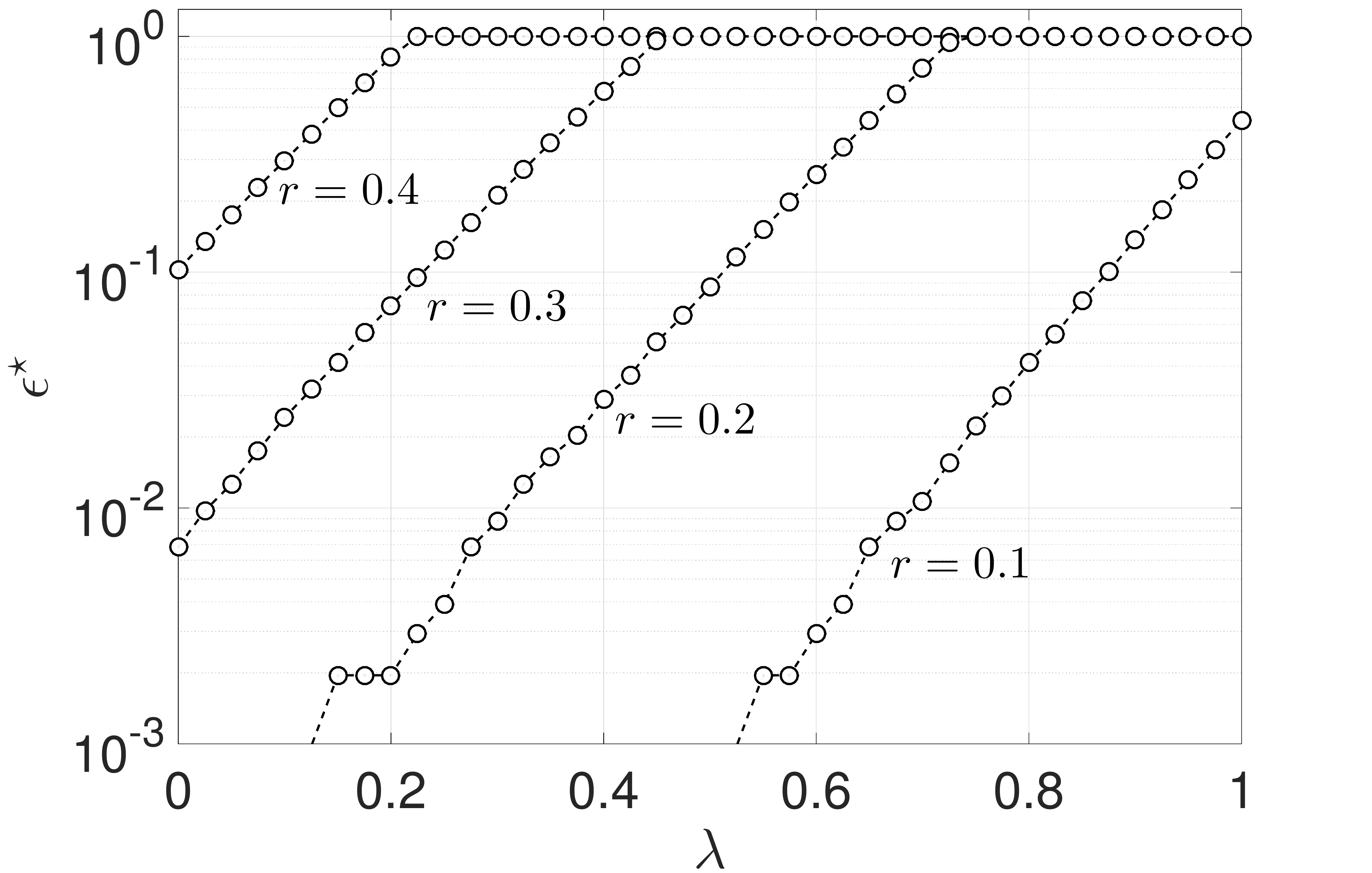}
\caption{$\epsilon^\star$ versus $\lambda$ for $r=0.1$, $0.2$, $0.3$, and $0.4$}
\label{fig:analysis}
\end{figure}

\subsection{Network Design}

Then we solve the network design for stabilization, i.e.,
Problem~\ref{prb:design:stbl}, using Theorem~\ref{thm:stabilization:iid}. Here
we assume that $r_{ij}$ depends only on $i$, i.e., $r_{ij} = r_i$ for all $j \in
[N]$. Under this assumption, we can choose the posynomial functions $\eta_{ij}$,
$\Phi_{ij}$, and $\Psi_{ij}$ satisfying \eqref{eq:def:Phi_i:iid} as $\eta_{ij} =
a_{ij} (\bar{\beta} - \ubar{\beta})$ and $\Phi_{ij} = \Psi_{ij} = a_{ij} r_i
(\bar{\beta} - \ubar{\beta})^2$ because we have
\begin{equation*}
\norm{A_{ij} - E[A_{ij}]} =
a_{ij} \max(r_i, 1-r_i)(\bar{\beta} - \ubar{\beta})
\end{equation*}
and $W(A_{ij}) = W(A_{ij}^\top) = a_{ij} r_i(1-r_i)(\bar{\beta} - \ubar{\beta})^2$.

We put the constraint $\epsilon \leq 0.2$, i.e., we require that the resulting
optimal parameter $r^\star$ guarantees stability of with probability at least
$0.8$. This constraint is equivalent to the monomial constraint:
\begin{equation*}
f_1 = \frac{\log(N/(0.2))}{\rho}\leq 1.
\end{equation*}
We use the cost function $R = \sum_{i=1}^N ({1}/{r_i})$. With these parameters,
we solve the geometric program \eqref{eq:gp:stbl:iid} and find the optimal
non-prevention probabilities~$r_1^\star$, $\dotsc$, $r_N^\star$.
Fig.~\ref{fig:synthesis} shows the value of the obtained $r^\star$ versus the
in-degree of the nodes. From the figure we can see that, the edges pointing
toward a node with the larger in-degree should receive protection resource with
the larger probability.

\begin{figure}
\centering
\includegraphics[width=.475\textwidth]{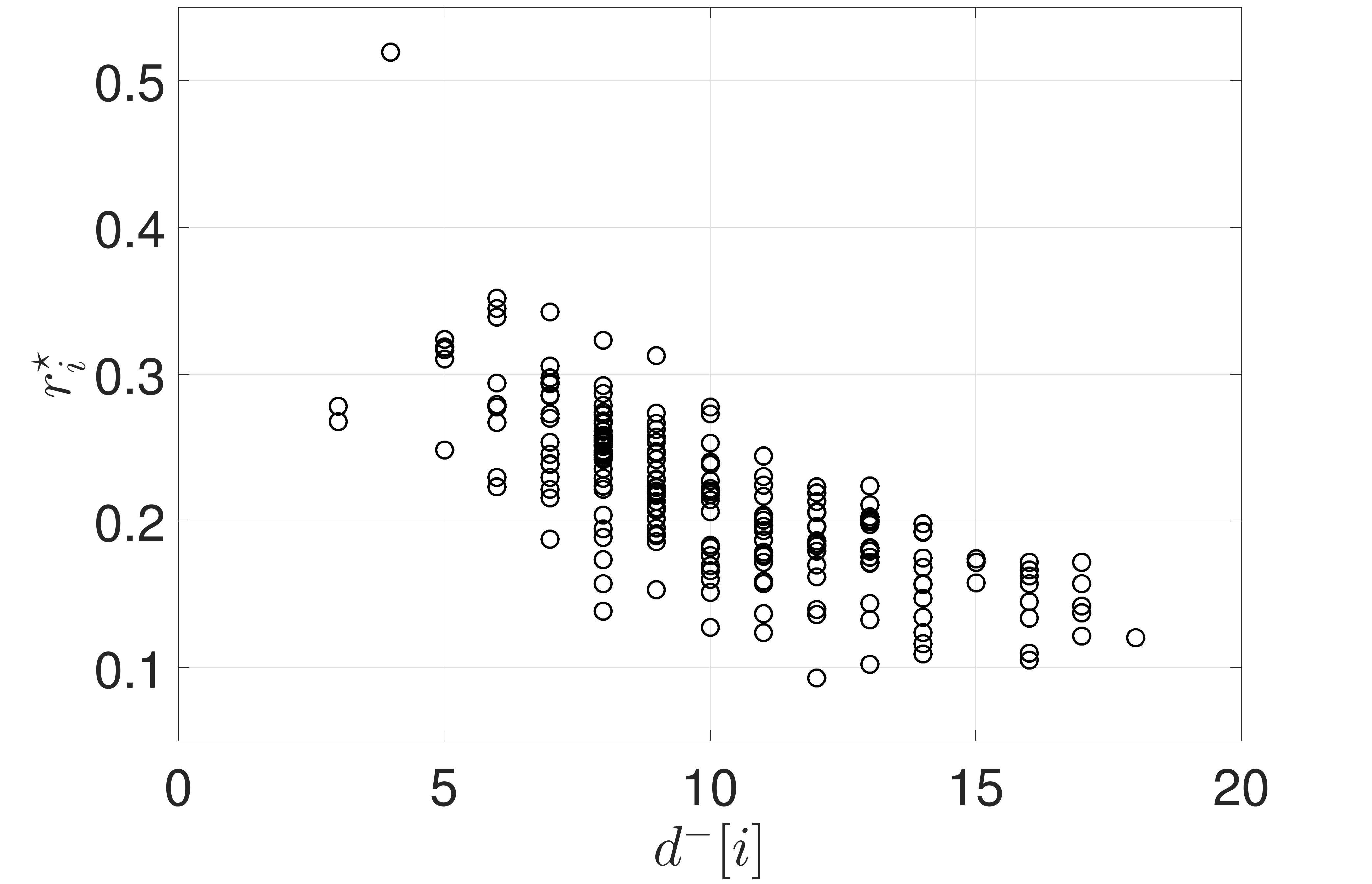}
\caption{In-degrees and optimal probabilities}
\label{fig:synthesis}
\end{figure}

\section{Conclusion}

We have studied the stability of the networks of positive linear
systems subject to time-invariant and random uncertainty. We have first presented
a collection of linear matrix inequalities to study the stability of the whole network around
the origin with a given probability and a decay rate. Based on this result, we have
then proposed a convex optimization framework to optimally design the
parameters of the probability distribution that describes the uncertainty of the
system. We have illustrated our results using a networked
susceptible-infected-susceptible viral spreading model.



\end{document}